\documentclass{amsart}
\usepackage{epsfig}
\usepackage{graphicx}
\usepackage{amscd}
\usepackage{amsmath}
\usepackage{amsxtra}
\usepackage{amsfonts}
\usepackage{amssymb}

\newtheorem{theorem}{Theorem}[section]
\newtheorem{corollary}[theorem]{Corollary}
\newtheorem{lemma}[theorem]{Lemma}
\newtheorem{proposition}[theorem]{Proposition}

\theoremstyle{definition}
\newtheorem{definition}[theorem]{Definition}
\newtheorem{remark}[theorem]{Remark}

\newtheorem{example}[theorem]{Example}
\theoremstyle{remark}

\renewcommand{\theclaim}{\textup{\theclaim}}

\newtheorem*{acknowledgements}{Acknowledgements}

\numberwithin{equation}{section}

\def\openone

{\mathchoice

{\hbox{\upshape \small1\kern-3.3pt\normalsize1}}

{\hbox{\upshape \small1\kern-3.3pt\normalsize1}}

{\hbox{\upshape \tiny1\kern-2.3pt\SMALL1}}

{\hbox{\upshape \Tiny1\kern-2pt\tiny1}}}

\makeatletter

\newbox\ipbox

\newcommand{\ip}[2]{\left\langle #1\, , \,#2\right\rangle}
\newcommand{\diracb}[1]{\left\langle #1\mathrel{\mathchoice

{\setbox\ipbox=\hbox{$\displaystyle \left\langle\mathstrut
#1\right.$}

\vrule height\ht\ipbox width0.25pt depth\dp\ipbox}

{\setbox\ipbox=\hbox{$\textstyle \left\langle\mathstrut
#1\right.$}

\vrule height\ht\ipbox width0.25pt depth\dp\ipbox}

{\setbox\ipbox=\hbox{$\scriptstyle \left\langle\mathstrut
#1\right.$}

\vrule height\ht\ipbox width0.25pt depth\dp\ipbox}

{\setbox\ipbox=\hbox{$\scriptscriptstyle \left\langle\mathstrut
#1\right.$}

\vrule height\ht\ipbox width0.25pt depth\dp\ipbox}

}\right. }

\newcommand{\dirack}[1]{\left. \mathrel{\mathchoice

{\setbox\ipbox=\hbox{$\displaystyle \left.\mathstrut
#1\right\rangle$}

\vrule height\ht\ipbox width0.25pt depth\dp\ipbox}

{\setbox\ipbox=\hbox{$\textstyle \left.\mathstrut
#1\right\rangle$}

\vrule height\ht\ipbox width0.25pt depth\dp\ipbox}

{\setbox\ipbox=\hbox{$\scriptstyle \left.\mathstrut
#1\right\rangle$}

\vrule height\ht\ipbox width0.25pt depth\dp\ipbox}

{\setbox\ipbox=\hbox{$\scriptscriptstyle \left.\mathstrut
#1\right\rangle$}

\vrule height\ht\ipbox width0.25pt depth\dp\ipbox}

} #1\right\rangle}

\newcommand{\bz}{\mathbb{Z}}

\newcommand{\br}{\mathbb{R}}

\def\blfootnote{\xdef\@thefnmark{}\@footnotetext}
\newcommand{\muf}{\widehat\mu_4}
\newcommand{\ty}{\emptyset}

\renewcommand{\mod}{\operatorname{mod}}

\input xy
\xyoption{all}
\usepackage{amssymb}

\newcommand{\lng}{\textup{lng}}

\def\-{^{-1}}

\begin{document}

\title[On the spectra of a Cantor measure]{On the spectra of a Cantor measure}
\author{Dorin Ervin Dutkay}
\blfootnote{Research supported in part by a grant from the National Science Foundation DMS-0704191.}
\address{[Dorin Ervin Dutkay] University of Central Florida\\
	Department of Mathematics\\
	4000 Central Florida Blvd.\\
	P.O. Box 161364\\
	Orlando, FL 32816-1364\\
U.S.A.\\} \email{ddutkay@mail.ucf.edu}
\author{Deguang Han}
\address{[Deguang Han] University of Central Florida\\
	Department of Mathematics\\
	4000 Central Florida Blvd.\\
	P.O. Box 161364\\
	Orlando, FL 32816-1364\\
U.S.A.\\}
\email{dhan@pegasus.cc.ucf.edu}
\author{Qiyu Sun}
\address{[Qiyu Sun] University of Central Florida\\
	Department of Mathematics\\
	4000 Central Florida Blvd.\\
	P.O. Box 161364\\
	Orlando, FL 32816-1364\\
U.S.A.\\}
\email{qsun@mail.ucf.edu}
\thanks{} 

\begin{abstract}
We analyze all orthonormal bases of exponentials on the Cantor set defined by Jorgensen and Pedersen in J. Anal. Math. 75,1998, pp 185-228. A complete characterization for all maximal sets of orthogonal exponentials is obtained by establishing a one-to-one correspondence with the spectral labelings of the infinite binary tree.  With the help of this characterization we obtain a sufficient condition for a spectral labeling to generate a spectrum (an orthonormal basis). This result not only provides us an easy and efficient way to construct various of  new spectra for the Cantor measure but also extends many previous results in the literature. In fact, most known examples of orthonormal bases of exponentials correspond to spectral labelings satisfying this sufficient condition.  We also obtain two new conditions for a labeling tree to generate a spectrum when other digits (digits not necessarily in $\{0, 1, 2, 3\}$) are used in the base 4 expansion of integers and  when bad branches are allowed in the spectral labeling. These new conditions yield new examples of spectra and in particular lead to a surprizing example which shows that a maximal set of orthogonal exponentials is not necessarily an orthonormal basis.
\end{abstract}
\maketitle \tableofcontents

\section{Introduction}\label{intr}

For certain probability measures $\mu$ in $\mathbb R^{d}$ there
exist orthonormal bases of countable families of complex
exponentials $\{e^{2\pi i\lambda\cdot x}\,|\, \lambda\in\Lambda\}$ for the Hilbert space  $L^{2}(\mu)$.  We called them Fourier series by analogy with the classical example of intervals on the real line. In this case,
the measure $\mu$ is called a {\it spectral measure} and the set
$\Lambda$ is called a {\it spectrum} for $\mu$. When
$\mu = \frac{1}{|\Omega|}\,dx$ (where $\Omega$ is bounded subset of
positive Lebesgue measure $|\Omega|>0$ and $dx$ is the Lebesgue measure),
the existence of a spectrum is closely related to the well-known
Fuglede conjecture which  asserts that there exists a spectrum for
$\mu$ if and only if $\Omega$ tiles $\mathbb R^{d}$ by translations using 
discrete set. This conjecture was proved to be false in higher
dimensions by Tao \cite{Tao04} and others, but it is still open in
dimension 1 and 2. We refer to \cite{MR2053982,MR1839575,MR1369421,MR1840101} for some
important results and developments related to the spectral pairs
with respect to probability measures that are obtained by restricting the
Lebesgue measure to  bounded sets.

\begin{definition}\label{defi1}
Let $e_\lambda(x):=e^{2\pi i\lambda\cdot x}$, $x\in\br^d,\lambda\in\br^d$. A probability measure $\mu$ on $\br^d$ is said to be a {\it spectral measure} if there exists a set $\Lambda\subset\br^d$ such that the family $\{e_\lambda\,|\,\lambda\in\Lambda\}$ is an orthonormal basis for $L^2(\mu)$. In this case $\Lambda$ is called a {\it spectrum} for the measure $\mu$.
\end{definition}

There exist other probability measures that are not the
restriction of the Lebesgue measure to  bounded sets, but they
admit spectra. The first example of a singular, non-atomic,
spectral measure was constructed by Jorgensen and Pedersen in
\cite{JoPe98}, and Strichartz \cite{Str98} gave a
simplification of part of the proof. These results led to the the
spectral theory for fractal measures which has recently become an
important topic of research in harmonic analysis. These fractal
measures also have very close connections with the theory of
multiresolution analysis in wavelet analysis (see e.g., \cite{MR2358531,MR2268116}).

The Jorgensen-Pedersen measure is constructed on a slight
modification of the Middle Third Cantor set. This can be obtained as follows: consider the interval $[0,1]$. Divide it into 4 equal intervals, and keep the intervals $[0,\frac14]$, and $[\frac12,\frac34]$. Then take each of these intervals and repeat the procedure ad inf. The result is a Cantor set
$$X_4:=\left\{\sum_{k=1}^\infty a_k\frac{1}{4^k}\,|\,\,a_k\in\{0,2\}\right\}.$$
The probability measure $\mu_4$ on $X_4$ assigns measure $\frac12$ to the sets $X_4\cap[0,\frac14]$ and $X_4\cap [\frac24,\frac34]$, measure $\frac14$ to the four intervals at the next stage, etc. It is the Hausdorff measure of Hausdorff dimension $\frac{\ln 2}{\ln 4}=\frac12$.

The set $X_4$ and the measure $\mu_4$ can be defined also in terms of iterated function systems (see \cite{Hut81} for details). Consider the iterated function system (IFS)
$$\tau_0(x)=\frac x4,\quad \tau_2(x)=\frac{x+2}{4},\quad(x\in\br).$$
Then the IFS $\{\tau_0,\tau_2\}$ has a unique attractor $X_4$, i.e., a unique compact subset of $\br$ with the property that
$$X_4=\tau_0(X_4)\cup\tau_2(X_4).$$
The measure $\mu_4$ is the unique probability measure on $\br$ which satisfies the invariance equation:
\begin{equation}\label{eqinv}
\int f(x)\,d\mu_4(x)=\frac{1}{2}\left(\int f\left(\frac{x}{4}\right)\,d\mu_4(x)+\int f\left(\frac{x+2}{4}\right)\,d\mu_4(x)\right),\quad(f\in C_c(\br)).
\end{equation}
Moreover, the measure $\mu_4$ is supported on $X_4$.

In \cite{JoPe98}, the authors proved that the set 
$$\Lambda=\left\{\sum_{k=0}^n 4^k d_k\,|\,d_k\in\{0,1\},n\geq0\right\}$$
is a spectrum for $\mu_4$. 

The results of Jorgensen and Pedersen were further extended for other measures, and new spectra were found in \cite{Str00,LaWa02,DuJo06,DuJo07a,DuJo07b,Li07a,Li07b}. Some surprising convergence properties of the associated Fourier series were discovered in \cite{Str06}.

 Two approaches to harmonic analysis on Iterated Function Systems have been popular: one based on a discrete version of the more familiar and classical second order Laplace differential operator of potential theory, see \cite{MR2246975,MR1840042};  and the other is based on Fourier series. The first model in turn is motivated by infinite discrete network of resistors, and the harmonic functions are defined by minimizing a global measure of resistance, but this approach does not rely on Fourier series. In contrast, the second approach begins with Fourier series, and it has its classical origins in lacunary Fourier series \cite{Kah86}.

In general, for a given probability measure $\mu$ any of the
following possibilities can occur: (i) there exists at most a finite number
of orthogonal complex exponentials in $L^{2}(\mu)$; (ii) there are
infinite families of orthogonal complex exponentials and one of them is
an orthonormal basis for  $L^{2}(\mu)$, and in this case $\mu$ is
a spectral measure. The first example satisfying (i) is  the
Middle Third Cantor set, with its Hausdorff measure of dimension
$\frac{\ln 2}{\ln 3}$ . In \cite{JoPe98} it was proved that for this
measure no three exponentials are mutually orhtogonal. Detailed
analysis on this was given and many new examples were constructed
in a recent paper \cite{DuJo07a}. However,  for a given measure $\mu$ it
remains a very difficult problem to ``characterize'' all the
spectra or the maximal families of orthogonal exponentials. Moreover, it is
not known whether every such a maximal family must be an orthonormal
basis for $L^{2}(\mu)$. The main purpose of  this paper is to
answer all these questions for the measure $\mu_4$.

In section 3 we first establish a one-to-one correspondence
between the labeling of the infinite binary tree and the base $4$
expansions (using the digits $\{0, 1, 2, 3\}$) of the integers. Then
we characterize all maximal sets of orthogonal exponentials in
$L^{2}(\mu_4)$ by showing that they correspond to spectral labelings (Definition \ref{defma1}) of the
binary tree (Theorem \ref{thma3}). In Example \ref{exsp1} we show that there are
maximal sets of orthogonal exponentials which are not spectra for
$\mu_4$. This is surprising, since in the previous examples in the
literature, all maximal sets of orthogonal exponentials were also
spectra for the associated fractal measure.

The spectral labeling characterization helps us obtain one
sufficient condition for a maximal family of exponentials  to an
orthonormal basis for $L^{2}(\mu_4)$ (Theorem \ref{thsp2}). This
sufficient condition improves the known results from \cite{JoPe98, Str00,
LaWa02, DuJo06}, and, as shown in Section 4, it clarifies why some of
the candidates for a spectrum constructed in \cite{ LaWa02, Str00} are
incomplete, and how they can be completed to spectra for
$\mu_{4}$.

 In Section 4 we consider
other digits that can be used for the base 4 expansion of the
integers in the candidate set $\Lambda$, and give some sufficient
conditions when these will generate spectra for $\mu_4$ (Theorem
\ref{tho1}). We construct some examples of spectra and give the
example showing that a maximal set of orthogonal exponentials is not
necessarily a spectrum. In addition a result of Strichartz in
\cite{Str00} is improved with the help of our Theorem \ref{tho1} (see Remark
\ref{remstr}).

In an attempt to obtain a ``complete'' characterization of all
the spectra, in section 5 we present a few other basic properties
of spectra for $\mu_4$ and give another sufficient condition for a
spectral labeling to generate a spectrum (Proposition \ref{propr2}) where
limited number of ``bad'' paths are allowed in the labeling. This
new condition allows us to construct an example (Example \ref{exr4}) of a
spectral labeling that gives us a spectrum even though it does
not satisfy the hypothesis of Theorem \ref{thsp2}. Although we were not
able to obtain a ``complete'' characterization for a maximal
family to generate a spectrum, we believe that a combination of our results Theorem \ref{thsp2} and Proposition \ref{propr2} might come close.

For the sake of clarity, in this paper we focus our discussion on the
fractal measure $\mu_{4}$. We believe that this example has many of the
key features that might occur in more general fractal measures,
and  most of our results can be generalized for other IFS
measures.

\section{Preliminaries}

To define the sets of integers that correspond to families of orthogonal exponentials, in this section we will recall some basic facts about base 4 expansions of integers.
\begin{definition}\label{defba1}
Let $k$ be an integer. Define inductively the sequences $(d_n)_{n\geq0}$ and $(k_n)_{n\geq0}$, with $d_n\in\{0,1,2,3\}$ and $k_n\in\bz$: $k_0:=k$; using division be 4 with remainder, there exist a unique $d_0\in\{0,1,2,3\}$ and 
$k_1\in\bz$ such that $k_0=d_0+4k_1$. If $k_n$ has been defined, then there exist a unique $d_{n}\in\{0,1,2,3\}$ and $k_n\in\bz$ such that $k_n=d_n+4k_{n+1}$. 

The infinite string $d_0d_1\dots d_n\dots$ will be called the {\it base 4 expansion} or the {\it encoding} of $k$. We will use the notation 
$$k=d_0d_1\dots d_n\dots.$$ 
We will denote by $\underline 0$ the infinite sequence $000\dots$, and similarly $\underline 3=333\dots$. The notation $d_0d_1\dots d_n\underline0$ indicates that the infinite string begins with $d_0\dots d_n$ and ends in an infinite repetition of the digit $0$. Similarly for the notation $d_0\dots d_n\underline3$.
\end{definition}

\begin{proposition}\label{propba2}
Let $k\in\bz$ with base 4 expansion $k=d_0\dots d_n\dots$. If $k\geq 0$ then its base 4 expansion ends in $\underline 0$, i.e., there exists $N\geq 0$ such that $d_n=0$ for all $n\geq N$. In this case 
\begin{equation}\label{eqba1}
k=d_0\dots d_N\underline 0=\sum_{n=0}^N4^nd_n.
\end{equation}  
 If $k<0$ then its base 4 expansion ends in $\underline 3$, i.e., there exists $N\geq 0$ such that $d_n=3$ for all $n\geq 3$. In this case 
 \begin{equation}\label{eqba2}
 k=d_0\dots d_n\underline 3=\sum_{n=0}^N4^nd_n-4^{N+1}.
 \end{equation}
  Moreover, if $k$ is defined by the formula on the right-hand side of \eqref{eqba1} or \eqref{eqba2} then its base 4 expansion is $d_0\dots d_N\underline0$, in the first case, or $d_0\dots d_N\underline 3$ in the second case. 
\end{proposition}

\begin{proof}
For $k\geq0$, the base 4 expansion is well known. Let us consider the case when $k<0$ and let $k=d_0\dots d_n\dots$ be its base 4 expansion. Take $N\geq0$ such that 
$k\geq -4^{N+1}$. Let $(k_n)_{n\geq 0}$ be defined as in Definition \ref{defba1}. Then $0>k_0=k\geq -4^{N+1}$. 
Since $k_1=\frac{k_0-d_0}{4}$ it follows that $k_1\geq\frac{-4^{N+1}-3}{4}\geq -4^N$. By induction $0>k_{N+1}\geq -4^0=-1$. So $k_{N+1}=-1$. Then $k_{N+2}=\frac{-1-3}{4}$, so $k_n=-1$ and $d_n=3$ for all $n\geq N+1$. Thus the base 4 expansion of $k$ ends in $\underline 3$. Moreover, since 
$k_{N+1}=-1$, we have that $k_{N}=d_{N}-4$, $k_{N-1}=d_{N-1}+4k_N=d_{N-1}+4d_N-4^2$, and, by induction
$$k=k_0=d_0+4d_1+\dots+ 4^Nd_N-4^{N+1}.$$
\end{proof}

\begin{lemma}\label{lemba3}
Let $b$ be an integer and let $b=b_0b_1\dots$ be its base 4 expansion. Let $a$ be another integer that has base 4 expansion ending with the expansion of $b$, i.e., $a=a_0\dots a_nb_0b_1\dots$. Then 
\begin{equation}\label{eqba3}
a=a_0+4a_1+\dots+4^na_n+4^{n+1}b.
\end{equation}
Conversely, if the integers $a$ and $b$ satisfy \eqref{eqba3} with $a_0\dots a_n\in\{0,1,2,3\}$, then the base $4$ expansion of $a$ has the form $a=a_0\dots a_nb_0b_1\dots$, where $b=b_0b_1\dots$ is the base $4$ expansion of $b$.

The base 4 expansion $d_0d_1\dots$ of an integer $k$ is completely determined by the conditions: $d_n\in\{0,1,2,3\}$ for all $n\geq0$, and 
$$\sum_{n=0}^Nd_n4^n\equiv k\mod 4^{N+1},\quad( N\geq0).$$
\end{lemma}

\begin{proof}
The proof follows directly from Proposition \ref{propba2} by a simple computation.
\end{proof}

\section{Main results}

In this section, we will characterize maximal sets of orthogonal exponentials and give a sufficient condition for such a maximal set to generate an orthonormal basis for $L^2(\mu_4)$.
\subsection{Maximal sets of orthogonal exponentials}
First we will characterize maximal sets of orthogonal exponentials. These will correspond to sets of integers whose base 4 expansions can be arranged in a binary tree. We will call this arrangement a {\it spectral labeling} of the binary tree.

\begin{definition}\label{defma1}
Let $\mathcal T$ be the complete infinite binary tree, i.e., the oriented graph that has vertices 
$$\mathcal V:=\{\ty\}\cup\left\{\epsilon_0\dots \epsilon_n\,|\, \epsilon_k\in\{0,1\}, n\geq0\right\},$$
and edges $\mathcal E$: $(\ty,0), (\ty,1)$, $(\epsilon_0\dots \epsilon_n,\epsilon_0\dots \epsilon_n\epsilon_{n+1})$ for all $\epsilon_0\dots \epsilon_n\in\mathcal V$, and $\epsilon_{n+1}\in\{0,1\}$, $n\geq0$. The vertex $\ty$ is the {\it root} of this tree.

A {\it spectral labeling} $\mathcal L$ of the binary tree is a labeling of the edges of $\mathcal T$ with labels in $\{0,1,2,3\}$ such that the following properties are satisfied:
\begin{enumerate}
\item For each vertex $v$ in $\mathcal V$, the two edges that start from $v$ have labels of different parity.
\item For each vertex $v$ in $\mathcal V$, there exist an infinite path in the tree that starts from $v$ and ends with edges that are all labeled $0$ or all labeled $3$. 
\end{enumerate}

We will use the notation $\mathcal T(\mathcal L)$ to indicate that we use the labeling $\mathcal L$.

Given a spectral labeling, we will identify the vertices $v\in\mathcal V$ with the finite word obtained by reading the labels of the edges in the unique path from the root $\ty$ to the vertex $v$. We will sometimes write $v=d_0d_1\dots d_n$, to indicate that the vertex $v$ is the one that is reached from the root by following the labels $d_0\dots d_n$.

We identify an {\it infinite path} in the tree $\mathcal T(\mathcal L)$ from a vertex $v$ with the infinite word obtained by reading the labels of the edges along this path. See Figure \ref{fig1} for the first few levels in a spectral labeling.
\end{definition}

\begin{figure}[ht]\label{fig1}
\centerline{ \vbox{\hbox{\epsfxsize 5cm\epsfbox{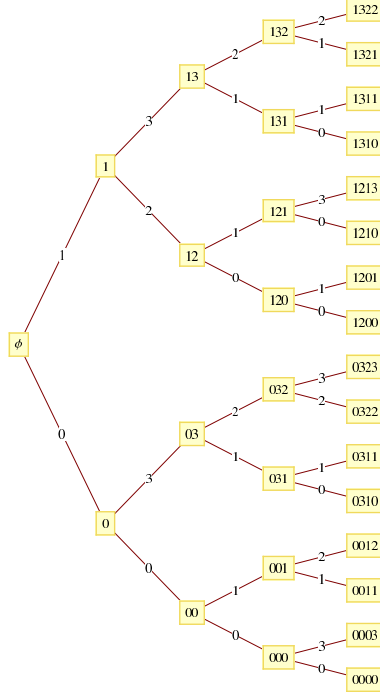}}}}
\caption{ The first levels in a spectral labeling of the binary tree. $0323$ is a path in the tree from the root $\ty$, $13$ is a path in the tree from the vertex $12$.}
\end{figure}

\begin{definition}\label{defma2}
Let $\mathcal L$ be a spectral labeling of the binary tree. Then {\it the set of integers associated to }$\mathcal L$ is the set
$$\Lambda(\mathcal L):=\left\{k=d_0d_1\dots...\,|\, d_0d_1\dots\mbox{ is an infinite path in the tree starting from }\ty\mbox{ and ending in }\underline0\mbox{ or }\underline 3\right\}.$$
\end{definition}

\begin{theorem}\label{thma3}
Let $\Lambda$ be a subset of $\br$ with $0\in\Lambda$. Then $\{e_\lambda\,|\,\lambda\in\Lambda\}$ is a maximal set of mutually orthogonal exponentials if and only if there exists a spectral labeling $\mathcal L$ of the binary tree such that $\Lambda=\Lambda(\mathcal L)$.
\end{theorem}

\begin{proof}

We will need several lemmas. 
\begin{lemma}\label{lemma4}
The Fourier transform of $\mu_4$ is 
\begin{equation}\label{eqma1}
\widehat\mu_4(t)=e^{\frac{2\pi i t}{3}}\prod_{j=1}^\infty \cos\left(2\pi \frac{t}{4^j}\right),\quad(t\in\br).
\end{equation}
The convergence of the infinite product is uniform on compact subsets of $\br$.
\end{lemma}
\begin{proof}
Applying the invariance equation \eqref{eqinv} to the exponential function $e_t$, $t\in\br$, we get
$$\widehat\mu_4(t)=\frac{1+e^{2\pi i2\frac{t}{4}}}{2}\widehat\mu_4(\frac t4)=e^{2\pi i\frac{t}{4}}\cos\left(2\pi \frac t4\right)\widehat\mu_4\left(\frac t4\right).$$
Since $\widehat\mu_4(0)=1$, the cosine function is Lipschitz near $0$, and $\cos 0=1$, we can iterate this relation to infinity and obtain 
$$\widehat\mu_4(t)=e^{2\pi i\sum_{j=1}^\infty\frac t{4^j}}\prod_{j=1}^\infty\cos\left(2\pi\frac t{4^j}\right).$$
\end{proof}

\begin{lemma}\label{lemma5}
Let $\lambda,\lambda'\in\br$. Then $e_\lambda$ is orthogonal to $e_{\lambda'}$ in $L^2(\mu_4)$ iff $\lambda-\lambda'\in\mathcal Z$, where
\begin{equation}\label{eqma2}
\mathcal Z:=\{x\in\br\,|\,\widehat\mu_4(x)=0\}=\{4^j(2k+1)\,|\,0\leq j\in\bz,k\in\bz\}.
\end{equation}
\end{lemma}

\begin{proof}
We have $\ip{e_\lambda}{e_{\lambda'}}=\int e^{2\pi i(\lambda-\lambda')x}\,d\mu_4(x)=\widehat\mu_4(\lambda-\lambda')$. So $e_\lambda\perp e_{\lambda'}$ iff $\lambda-\lambda'\in\mathcal Z$. 
Using the infinite product in \eqref{eqma2}, we obtain that $\lambda-\lambda'\in\mathcal Z$ iff there exists $j\geq 1$ such that $\cos\left(2\pi \frac{\lambda-\lambda'}{4^j}\right)=0$. So $2\pi(\lambda-\lambda')\in 4^j\pi(\bz+\frac12)$. This implies \eqref{eqma2}.
\end{proof}

Note that, since $0\in\Lambda$, for any element $a\in\Lambda$, we have $e_a\perp e_0$. Then with Lemma \ref{lemma5}, we must have $a\in\mathcal Z\subset\bz$.

We will use the following notation:
for an integer $k$ with base $4$ expansion $k=d_0\dots d_n\dots$, we will denote by $d_n(k):=d_n$, the $n$-th digit of the base 4 expansion of $k$.

The next lemma follows from an easy computation.
\begin{lemma}\label{lemma6}
If $n,n'\geq0$, $k,k,a,a'\in\bz$ with $a,a'$ not divisible by 4, and $4^n(4k+a)=4^{n'}(4k'+a')$ then $n=n'$.
\end{lemma}

\begin{lemma}\label{lemtwo}
Let $\Lambda$ be a subset of $\br$ with $0\in\Lambda$. Assume $\{e_\lambda\,|\,\lambda\in\Lambda\}$ is a maximal set of orthogonal exponentials in $L^2(\mu_4)$. Then for $d_0,\dots,d_{n-1}\in\{0,1,2,3\}$ the set 
$$D(d_0\dots d_{n-1}):=\{d_n(a)\,|\, a\in\Lambda, d_0(a)=d_0,\dots, d_{n-1}(a)=d_{n-1}\}$$
has either zero or two elements of different parity. This means that the $n$-th digit of the base 4 expansion of elements in $\Lambda$ with prescribed first $n-1$ digits, can take only $0$ or $2$ values, and if it takes 2 values,  then these values must have different parity, i.e., $\{0,1\}$, or $\{0,3\}$, or $\{1,2\}$ or $\{2,3\}$.

\end{lemma}

\begin{proof}
Suppose $D(d_0\dots d_{n-1})$ has at least one element. Suppose $a,a'\in D(d_0\dots d_{n-1})$ with $d_k(a)=d_k(a')=d_k$ for all $0\leq k\leq n-1$, and assume $d_n(a)\neq d_n(a')$. 

Then (see Lemma \ref{lemba3}) there exist $b,b'\in\bz$ such that
$$a=4^{n+1}b+4^nd_n(a_0)+4^{n-1}d_{n-1}+\dots+d_0,\quad a'=4^{n+1}b'+4^nd_n(a_0')+4^{n-1}d_{n-1}+\dots+d_0$$

Then $a-a'=4^n(4(b-b')+d_n(a)-d_n(a'))$. By Lemma \ref{lemma5}, since $a,a'\in\Lambda$, we must have $a-a'\in\mathcal Z$, so $a-a'=4^m(2k+1)=4^m(4l+e)$ for some $m\geq0$, $k,l\in\bz$, $e\in\{1,3\}$. Thus, with Lemma \ref{lemma6}, $n=m$ and $d_n(a)-d_n(a')$ is an odd number. In particular, it follows that $D(d_0\dots d_{n-1})$ contains at most 2 elements.

Suppose now that $D(d_0\dots d_{n-1})$ has just one element. Then for all $a\in\Lambda$, with $d_k(a)=d_k$ for all $0\leq k\leq n-1$, one has that $d_n(a)$ is constant $d_n$.

Let $d_n':=d_n+1\mod 4$ and let $a':=4^nd_n'+4^{n-1}d_{n-1}+\dots+d_0$. We claim that $e_{a'}$ is orthogonal to all $e_{a}$, $a\in\Lambda$.

Let $a\in\Lambda$.

Case I: $d_k(a)=d_k$ for all $0\leq k\leq n-1$. Then, with Lemma \ref{lemba3}, for some $b\in \bz$,
$$a=4^{n+1}b+4^nd_n+4^{n-1}d_{n-1}+\dots+d_0$$
so
$a-a'=4^n(4b+d_{n}-d_{n}')\in 4^n(2\bz+1)\subset\mathcal Z$. Therefore, with Lemma \ref{lemma5}, $e_{a'}\perp e_{a}$.

Case II: There is an integer $0\leq k\leq n-1$ such that $d_0(a)=d_0,\dots, d_{k-1}(a)=d_{k-1}$ and $d_k(a)\neq d_k$. Then for some $b\in\bz$,
$$a=4^{k+1}b+4^kd_k(a)+4^{k-1}d_{k-1}+\dots+d_0$$
Since $D(d_0\dots d_{n-1})$ is not empty, there is a $a''\in\Lambda$ such that $d_0(a'')=d_0,\dots,d_k(a'')=d_k$, so 
$$a''=4^{k+1}b''+4^kd_k+4^{k-1}d_{k-1}+\dots+d_0,$$
for some $b''\in\bz$. 

Then, as before, since $a,a''$ are in the tree, and they differ first time at the $k$-th digit, we have that $d_k-d_k(a)$ is odd.

It follows that $a-a'=4^k(4b-4^{n-k}d_n'-4^{n-k-1}d_{n-1}-\dots-4d_{k+1}+d_k(a)-d_k)\in 4^k(2\bz+1)\subset \mathcal Z$. Hence $e_b\perp e_{a'}$.

\end{proof}

We construct the spectral labeling $\mathcal L$ as follows: we label the root of the tree by $\emptyset$. Using Lemma \ref{lemtwo}, the set $D(\emptyset):=\{d(a_0)\,|\, a_0\in\Lambda\}$ has two elements $d_0$ and $d_0'$. We label the edges from $\emptyset$ by $d_0$ and $d_0'$.

By induction, if we constructed the label $d_0\dots d_n$ for a vertex, this means that there exists an element $a$ of $\Lambda$ that has base $4$ expansion starting with $d_0\dots d_n$. Therefore, using Lemma \ref{lemtwo}, the set $D(d_0\dots d_n)$ contains exactly two elements of different parity $e,e'$. We label the edges that start from the vertex $d_0\dots d_n$ by these elements $e,e'$. In particular we have that the sets 
$D(d_0\dots d_ne)$ and $D(d_0\dots d_ne')$ are not empty. 

Next, we check that, from any vertex in this tree, there exists an infinite path that ends in $\underline 0$ or $\underline 3$.

Consider a vertex in this tree, and let $d_0\dots d_{n}$ be its label. Then, by construction, the set $D(d_0\dots d_{n})$ is not empty. Therefore there is some $a$ in $\Lambda$ such that $d_0(a)=d_0,\dots,d_{n}(a)=d_{n}$. If we denote $d_k:=d_k(a)$ for all $k\geq n$, then by construction the tree contains the vertices labeled 
$d_0\dots d_k$ for all $k\geq0$. Since the string $d_0d_1\dots$ is the base $4$ expansion of $a$, it follows that the infinite sequence $d_0d_1\dots$ ends in either $\underline 0$ or $\underline 3$. Therefore there is an infinite path from the vertex $d_0\dots d_n$ that ends in either $\underline 0$ or $\underline 3$.

Finally, we have to check that $\Lambda=\Lambda(\mathcal L)$. If $a\in\Lambda$ and it has base 4 expansion $a=d_0d_1\dots$, then the vertices $d_0\dots d_k$ are all in the tree $\mathcal T(\mathcal L)$ so the infinite path $d_0d_1\dots$ is a path in this tree starting from the root $\emptyset$. Thus $\Lambda\subset \Lambda(\mathcal L)$.

For the converse we prove the following: 
\begin{lemma}\label{lemto}
If $a=d_0d_1\dots, a'=d_0'd_1'\dots$ are two distinct infinite paths in the binary tree $\Lambda(\mathcal L)$ starting from the root, that end in either $\underline 0$ or $\underline 3$, then $e_{a}\perp e_{a'}$. 
\end{lemma}

\begin{proof}
Let $k\geq 0$ be the first index such that $d_k\neq d_k'$. Then $d_0=d_0',\dots,d_{k-1}=d_{k-1}'$ and since $\mathcal L$ is a spectral labeling, we have that $d_k-d_k'$ is odd. With Lemma \ref{lemba3} there exist $b,b'\in\bz$ such that
$$a=4^{k+1}b+4^kd_k+4^{k-1}d_{k-1}+\dots+d_0,\quad a'=4^{k+1}b'+4^kd_k'+4^{k-1}d_{k-1}'+\dots+d_0'.$$
Then $a-a'=4^k(4(b-b')+d_k-d_k')\in 4^k(2\bz+1)\subset \mathcal Z$. So $e_{a}\perp e_{a'}$.
\end{proof}

Lemma \ref{lemto} shows that, since $\mathcal L$ is a spectral labeling, the set $\{e_\lambda\,|\,\lambda\in\Lambda(\mathcal L)\}$ is a set of mutually orthogonal exponentials. Since $\Lambda\subset \Lambda(\mathcal L)$ and $\Lambda$ is maximal, it follows that $\Lambda=\Lambda(\mathcal L)$.

It remains to prove that, if $\mathcal L$ is a spectral labeling, then $\Lambda(\mathcal L)$ corresponds to a maximal set of exponentials. We have seen above that $\Lambda(\mathcal L)$ corresponds to a family of orthogonal exponentials; we have to prove it is maximal. 
 Suppose there exists $\lambda\in\br$ such that $e_\lambda\perp e_a$ for all $a\in\Lambda(\mathcal L)$. 
 In particular $e_\lambda\perp e_0$, and with Lemma \ref{lemma5}, we have $\lambda\in\bz$. 
 Let $d_0d_1\dots$ be the base 4 expansion of $\lambda$. Let $k\geq0$ be the first index such that $d_0\dots d_{k}$ is not in the tree $\mathcal T(\mathcal L)$. One of the labels of the edges from the vertex $d_0\dots d_{k-1}$ has the same parity as $d_k$, and is different from $d_k$. Let $d_k'$ be this label. Then $d_k-d_k'\in\{-2,2\}$. Using property (ii) in the definition of a spectral labeling, there exists an infinite path $a$ in the tree that starts with $d_0\dots d_{k-1}d_k'$ and ends with $\underline 0$ or $\underline3$. Then,  
 $$a=d_0+\dots+4^{k-1}d_{k-1}+4^kd_k+4^{k+1}b,\quad\lambda=d_0+\dots+4^{k-1}d_{k-1}+4^kd_k'+4^{k+1}b',$$
for some $b,b'\in\bz$. Then $a-\lambda=4^{k}(d_k-d_k'+4(b-b'))\not\in\mathcal Z$, because $d_k-d_k'$ is even, and not a multiple of $4$ (see Lemma \ref{lemma6}). With Lemma \ref{lemma5}, $e_\lambda$ is not perpendicular to $e_a$. 
This shows that $\Lambda(\mathcal L)$ corresponds to a maximal set of orthogonal exponentials.

This concludes the proof of Theorem \ref{thma3}.
\end{proof}

\subsection{Spectral sets}
Theorem \ref{thma3} shows that when a spectral labeling $\mathcal L$ of the binary tree is given, it generates a maximal family of mutually orthogonal exponentials, by reading base 4 expansions from the tree. 
In this section we will give a sufficient condition for a spectral labeling to generate a spectral set, i.e., an orthonormal basis of exponentials.

We will begin by defining certain ``good'' paths. The restriction on the spectral labeling will require that good paths can be found from any vertex.

\begin{definition}\label{defsp1}
Let $a\in\bz$ and let $a=d_0d_1\dots$ be its base 4 expansion. We call the {\it length} of $a$ the smallest integer $n$ such that either $d_k=0$ for all $k\geq n$ or $d_k=3$ for all $k\geq n$. We will use the notation $n=\lng(a)$.

Fix integers $P,Q>0$. Let $\omega=\omega_0\omega_1\dots$ be an infinite path ending in $\underline 0$ or $\underline 3$,  $\omega_n\in\{0,1,2,3\}$ for all $n\geq 0$. We will say that the path $\omega$ is {\it $(P,Q)$-good} (or just {\it good}) if the there exists $n\geq 0$ such that the following two conditions are satisfied:
\begin{enumerate}
\item
$\omega_0,\dots,\omega_n\in\{0,2\}$ and the number of occurrences of $2$ in $\omega_0\dots\omega_n$ is less than $P$;
\item $\lng(\omega_{n+1}\omega_{n+2}\dots)\leq Q$.
\end{enumerate}
\end{definition} 

\begin{theorem}\label{thsp2}
Let $\mathcal L$ be a spectral labeling of the binary tree. Suppose there exist integers $P,Q\geq0$ such that for any vertex $v$ in the tree, there exists a $(P,Q)$-good path starting from the vertex $v$. Then the set $\Lambda(\mathcal L)$ is a spectrum for $\mu_4$.
\end{theorem}

 We divide the proof into several lemmas.

\begin{lemma}[\cite{JoPe98}]\label{lemsp3}
Let $\Lambda$ be a set such that $\{e_\lambda\,|\,\lambda\in\Lambda\}$ is an orthonormal family in $L^2(\mu_4)$. Then 
\begin{equation}\label{eqsp1}
\sum_{\lambda\in\Lambda}|\widehat\mu_4(t+\lambda)|^2\leq 1\quad(t\in\br).
\end{equation}
The set $\Lambda$ is a spectrum for $\mu_4$ iff
\begin{equation}\label{eqsp2}
\sum_{\lambda\in\Lambda}|\muf(t+\lambda)|^2=1\quad(t\in\br).
\end{equation}
\end{lemma}

\begin{proof}
Let $\mathcal P$ be the projection onto the span of $\{e_\lambda\,|\,\lambda\in\Lambda\}$. Then, using Parseval's identity, we have for all $t\in\br$:
$$1\geq \|\mathcal Pe_{-t}\|^2=\sum_{\lambda\in\Lambda}|\ip{e_\lambda}{e_{-t}}|^2=\sum_{\lambda\in\Lambda}|\muf(t+\lambda)|^2.$$
This implies \eqref{eqsp1} and one of the $\Rightarrow$ part in the last statement. For the converse, if \eqref{eqsp2}
holds, then $e_{-t}$ is in the span of $\{e_\lambda\}_\lambda$, and using the Stone-Weiertrass theorem, this implies that the span is $L^2(\mu_4)$.
\end{proof}

\begin{lemma}\label{lemsp4}
Assume that there exist $\epsilon_0>0$ and $\delta_0>0$ such that for any $y\in[-\epsilon_0,1+\epsilon_0]$ and any vertex $v=d_0\dots d_{N-1}$ in the binary tree $\mathcal T(\mathcal L)$, there exists an infinite path $\lambda(d_0\dots d_{N-1})$ in the tree, starting from $v$, ending in $\underline 0$ or $\underline 3$, such that $|\muf(y+\lambda(d_0\dots d_{N-1}))|^2\geq\delta_0$. Then $\Lambda(\mathcal L)$ is a spectrum for $\mu_4$.
\end{lemma}

The main idea of the proof of Lemma \ref{lemsp4} is the same as the one used in a characterization of orthonormal scaling functions in wavelet theory \cite{DGH00}, and is similar to the one used in the proof of Theorem 2.8 in \cite{Str00}. But since 0 is not always present in the branching at a vertex, the details are more complicated.  
\begin{proof}[Proof of Lemma \ref{lemsp4}]
With Theorem \ref{thma3} we know that $\{e_\lambda\,|\,\lambda\in\Lambda(\mathcal L)\}$ is an orthonormal family.
We need to check \eqref{eqsp2}. For a finite word $d_0\dots d_{N-1}$ with $d_0,\dots d_{N-1}\in\{0,1,2,3\}$, we write $d_0\dots d_{N-1}\in\Lambda(\mathcal L)$, if $d_0\dots d_{N-1}$ is the label of a vertex in the binary tree $\mathcal T(\mathcal L)$. 

For $d_0\dots d_{N-1}$ in $\Lambda(\mathcal L)$, let 
\begin{equation}\label{eqsp3}
P_x^N(d_0\dots d_{N-1}):=\prod_{j=1}^{N}\cos^2\left(\frac{2\pi(x+d_0+\dots +4^{N-1}d_{N-1})}{4^j}\right),\quad(x\in\br).
\end{equation}

We claim that for any $N\geq1$,
\begin{equation}\label{eqsp4}
\sum_{d_0\dots d_{N-1}\in\Lambda(\mathcal L)}P_x^N(d_0\dots d_{N-1})=1.
\end{equation}

For this, note that if $\{e,e'\}$ is any one of the following sets $\{0,1\}$, $\{0,3\}$, $\{1,2\}$, $\{2,3\}$, we have
\begin{equation}\label{eqsp5}
\cos^2\left(\frac{2\pi(x+e)}{4}\right)+\cos^2\left(\frac{2\pi(x+e')}{4}\right)=1,\quad(x\in\br).
\end{equation}
Then \eqref{eqsp4} follows from \eqref{eqsp5} by induction.

Next, fix $x\in\br$. Pick $Q_1$ such that for $N\geq Q_1$, $\frac{|x|}{4^N}\leq\epsilon_0$. Then for any $d_0\dots d_{N-1}\in\Lambda(\mathcal L)$, the point 
$y:=\frac{x+d_0+\dots+4^{N-1}d_{N-1}}{4^N}\in [-\epsilon_0,1+\epsilon_0]$. Therefore there exists a path $\lambda(d_0\dots d_{N-1})$ starting from the vertex $d_0\dots d_{N-1}$, ending in $\underline 0$ or $\underline 3$ with $|\muf(y+\lambda(d_0\dots d_{N-1}))|^2\geq \delta_0$. We have 
$$
P_x^N(d_0\dots d_{N-1})\leq\frac{1}{\delta_0}P_x^N(d_0\dots d_{N-1})\left|\muf\left(\frac{x+d_0+\dots+4^{N-1}d_{N-1}}{4^N}+\lambda(d_0\dots d_{N-1})\right)\right|^2$$
$$=\frac{1}{\delta_0}\prod_{j=1}^N\cos^2\left(\frac{2\pi(x+d_0+\dots+4^{N-1}d_{N-1}+4^N\lambda(d_0\dots d_{N-1}))}{4^j}\right)\times$$
$$ \prod_{j=1}^\infty\left|\cos^2\left(\frac{x+d_0+\dots+4^{N-1}d_{N-1}+4^{N}\lambda(d_0\dots d_{N-1})}{4^{N+j}}\right)\right|^2$$
$$=\frac{1}{\delta_0}\left|\muf\left(x+d_0+\dots+4^{N-1}d_{N-1}+4^N\lambda(d_0\dots d_{N-1})\right)\right|^2
=\frac{1}{\delta_0}|\muf(x+\eta_x(d_0\dots d_{N-1}))|^2,$$
where for all $d_0\dots d_{N-1}\in\Lambda(\mathcal L)$, we denote
$$\eta_x(d_0\dots d_{N-1}):=d_0+\dots+4^{N-1}d_{N-1}+4^N\lambda(d_0\dots d_{N-1})\in\Lambda(\mathcal L).$$
Note that the base 4 expansion of $\eta_x(d_0\dots d_{N-1})$ starts with $d_0\dots d_{N-1}$.

We claim that for any $\epsilon>0$ there exists $P_\epsilon$ and $Q_\epsilon$ such that 
\begin{equation}\label{eqsp6}
\sum_{\stackrel{d_0\dots d_{N-1}\in\Lambda(\mathcal L)}{\lng(\eta_x(d_0\dots d_{N-1}))\geq P_\epsilon}} P_x^N(d_0\dots d_{N-1})<\epsilon,\quad(N\geq Q_\epsilon).
\end{equation}

Fix $\epsilon>0$. Using \eqref{eqsp1}, there exists $P_\epsilon\geq Q_1=:Q_\epsilon$ such that 
$$\sum_{\lambda\in\Lambda(\mathcal L),\lng(\lambda)\geq P_\epsilon}|\muf(x+\lambda)|^2<\epsilon\delta_0.$$
Then, using the previous calculation, for $N\geq Q_\epsilon$,
$$\sum_{\stackrel{d_0\dots d_{N-1}\in\Lambda(\mathcal{L})}{\lng(\eta_x(d_0\dots d_{N-1}))\geq P_\epsilon}}P_x^N(d_0\dots d_{N-1})\leq\frac{1}{\delta_0}\sum_{\stackrel{d_0\dots d_{N-1}\in\Lambda(\mathcal{L})}{\lng(\eta_x(d_0\dots d_{N-1}))\geq P_\epsilon}}\left|\muf\left(x+\eta_x(d_0\dots d_{N-1})\right)\right|^2
$$$$\leq \frac{1}{\delta_0}\sum_{\lambda\in\Lambda(\mathcal L),\lng(\lambda)\geq P_\epsilon}|\muf(x+\lambda)|^2<\epsilon.$$
This proves \eqref{eqsp6}.

From \eqref{eqsp6} we get that for all $N\geq Q_\epsilon$, 
\begin{equation}\label{eqsp6_1}
\sum_{\stackrel{d_0\dots d_{N-1}\in\Lambda(\mathcal{L})}{\lng(\eta_x(d_0\dots d_{N-1}))< P_\epsilon}}P_x^N(d_0\dots d_{N-1})=
\sum_{d_0\dots d_{N-1}\in\Lambda(\mathcal{L})}P_x^N(d_0\dots d_{N-1})-
\sum_{\stackrel{d_0\dots d_{N-1}\in\Lambda(\mathcal{L})}{\lng(\eta_x(d_0\dots d_{N-1}))\geq P_\epsilon}}P_x^N(d_0\dots d_{N-1})\end{equation}
\begin{equation}\label{eqsp6_2}
\stackrel{\mbox{by \eqref{eqsp4}}}{=}1-\sum_{\stackrel{d_0\dots d_{N-1}\in\Lambda(\mathcal{L})}{\lng(\eta_x(d_0\dots d_{N-1}))\geq P_\epsilon}}P_x^N(d_0\dots d_{N-1})>1-\epsilon.
\end{equation}

We also have for all $\lambda=d_0d_1\dots\in\Lambda(\mathcal L)$, 
\begin{equation}\label{eqsp7}
|\muf(x+\lambda)|^2=\lim_{N\rightarrow\infty}P_x^N(d_0\dots d_{N-1}).
\end{equation}
To prove \eqref{eqsp7}, we consider two cases: if $\lambda$ ends in $\underline 0$, then $\lambda=d_0+\dots+4^{p-1}d_{p-1}$ for some $p\geq0$, $d_k=0$ for $k\geq p$, and for $N\geq p$,
$$P_x^N(d_0\dots d_{N-1})=\prod_{j=1}^N\cos^2\left(\frac{2\pi(x+\lambda)}{4^j}\right)\rightarrow|\muf(x+\lambda)|^2.$$
If $\lambda$ ends in $\underline 3$, then $\lambda=d_0+\dots 4^{p-1}d_{p-1}-4^p$, for some $p$, $d_k=3$ for $k\geq p$, and for $p\geq N$,
$$P_x^N(d_0\dots d_{N-1})=\prod_{j=1}^N\cos^2\left(\frac{2\pi(x+d_0+\dots+4^{p-1}d_{p-1}+4^p(3+\dots+3\cdot 4^{N-1-p}))}{4^j}\right)=$$
$$\prod_{j=1}^N\cos^2\left(\frac{2\pi(x+d_0+\dots+4^{p-1}d_{p-1}-4^{p}+4^N)}{4^j}\right)=\prod_{j=1}^N\cos^2\left(\frac{2\pi(x+\lambda)}{4^j}\right)\rightarrow|\muf(x+\lambda)|^2.$$
This proves \eqref{eqsp7}.

Now, any $\lambda\in\Lambda(\mathcal L)$ with $\lng(\lambda)<P_\epsilon$ has base 4 expansion of the form $\lambda=d_0\dots d_{P_\epsilon-1}\underline0$ or $\lambda=d_0\dots d_{P_\epsilon-1}\underline3$, with $d_0\dots d_{P_\epsilon}\in\Lambda(\mathcal L)$. Therefore there are at most $2^{P_\epsilon}\cdot 2=2^{P_\epsilon+1}$ such $\lambda$. With \eqref{eqsp7}, for each such $\lambda$ we can approximate $|\muf(x+\lambda)|^2$ by 
$P_x^N(d_0(\lambda)\dots d_{N-1}(\lambda))$, where $d_0(\lambda)d_1(\lambda)\dots$ is the base 4 expansion of $\lambda$.

Therefore, using \eqref{eqsp7}, there exists $N$ as large as we want, $N\geq Q_\epsilon$, such that 
\begin{equation}\label{eqsp8}
\sum_{\lambda\in\Lambda(\mathcal L),\lng(\lambda)<P_\epsilon}|\muf(x+\lambda)|^2 >\sum_{\lambda\in\Lambda(\mathcal L),\lng(\lambda)<P_\epsilon}P_x^N(d_0(\lambda)\dots d_{N-1}(\lambda))-\epsilon.
\end{equation}

But if $d_0\dots d_{N-1}\in\Lambda(\mathcal L)$ and $\eta:=\eta_x(d_0\dots d_{N-1})$ has length $\lng(\eta)<P_\epsilon$ then, the first $N$ digits of $\eta_x(d_0\dots d_{N-1})$ are $d_0(\eta)=d_0,\dots,d_{N-1}(\eta)=d_{N-1}$ and $\eta_x(d_0\dots d_{N-1})$ is an element of $\Lambda(\mathcal L)$ such that $\lng(\eta)<P_\epsilon$. Therefore
\begin{equation}\label{eqsp9}
\sum_{\stackrel{d_0\dots d_{N-1}\in\Lambda(\mathcal L)}{\lng(\eta_x(d_0\dots d_{N-1}))<P_\epsilon}}P_x^N(d_0\dots d_{N-1})\leq
\sum_{\lambda\in\Lambda(\mathcal L),\lng(\lambda)<P_\epsilon}P_x^N(d_0(\lambda)\dots d_{N-1}(\lambda)).
\end{equation}
From \eqref{eqsp9}, and \eqref{eqsp6_1}, \eqref{eqsp6_2} we get 

\begin{equation}\label{eqsp10}
\sum_{\lambda\in\Lambda(\mathcal L),\lng(\lambda)<P_\epsilon}P_x^N(d_0(\lambda)\dots d_{N-1}(\lambda))>1-\epsilon.
\end{equation}
Then using \eqref{eqsp8}, we have
$$\sum_{\lambda\in\Lambda(\mathcal L)}|\muf(x+\lambda)|^2\geq\sum_{\lambda\in\Lambda(\mathcal L),\lng(\lambda)<P_\epsilon}|\muf(x+\lambda)|^2>1-2\epsilon.$$
Since $\epsilon>0$ and $x\in\br$ are arbitrary, Lemma \ref{lemsp4} follows from Lemma \ref{lemsp3}.
\end{proof}

\begin{lemma}\label{lemsp5}
For each $P,Q\geq0$, there exists $\delta>0$ depending only on $P,Q$, such that for all $x\in[-\frac14,\frac34]$ and all $(P,Q)$-good paths $\omega$ of one of the forms $\omega=\underline 0$ or $\omega=0\dots 02d_0d_1\dots$, the following inequality holds
$$|\muf(x+\omega)|^2\geq\delta.$$
(Note that, unless it is $\underline0$, the path $\omega$ contains at least one $2$ after some zeros. The $2$ can be on the first position $2\dots$. Note also that the path does not have to be in the binary tree.)
\end{lemma}
\begin{proof}
First we prove that for any $n,k\in\bz$, $n\geq0$,
\begin{equation}\label{eqsp11}
|\muf(x+4^nk)|^2\geq|\muf(x)|^2\left|\muf\left(\frac{x}{4^n}+k\right)\right|^2,\quad(x\in\br).
\end{equation}

If $n\geq 1$, we have
$$|\muf(x+4^nk)|^2=\cos^2\left(\frac{2\pi (x+4^nk)}{4}\right)\dots\cos^2\left(\frac{2\pi (x+4^nk)}{4^n}\right)\prod_{j=n+1}^\infty\cos^2\left(\frac{2\pi (x+4^nk)}{4^j}\right)=$$
$$\prod_{j=1}^n\cos^2\left(\frac{2\pi x}{4^j}\right)\prod_{j=1}^\infty\cos^2\left(\frac{2\pi (\frac{x}{4^n}+k)}{4^j}\right)\geq|\muf(x)|^2\left|\muf\left(\frac{x}{4^n}+k\right)\right|^2.$$
If $n=0$, then $|\muf(x+4^0k)|^2\geq|\muf(x)|^2|\muf(\frac{x}{4^0}+k)|^2$ simply because $|\muf(x)|\leq 1$. 
This proves \eqref{eqsp11}.

The function $|\muf|^2$ is continuous and its zeros are $\mathcal Z=\{4^j(2k+1)\,|\,j\geq 0, j,k\in\bz\}$ (see Lemma \ref{lemma5}). This implies in particular that 
$|\muf(4k+2)|^2\neq0$ for all $k\in\bz$. 

If an integer $a$ has base 4 expansion $a=a_0a_1\dots$ of length $\lng(a)\leq Q$ then $|a|\leq 4^Q$. Indeed, if $a=a_0\dots a_{Q-1}\underline0$, then 
$0\leq a=a_0+\dots+4^{Q-1}a_{Q-1}\leq 3+\dots + 4^{Q-1}3=4^Q-1$. If $a=a_0\dots a_{Q-1}\underline 3$, then $0\geq a=a_0+\dots +4^{Q-1}a_{Q-1}-4^Q\geq -4^Q$. 

Pick $\epsilon_1>0$ small (we will need $\epsilon_1<\frac{7}{48}$). The function $|\muf|^2$ is continuous and non-zero on the compact set
$$A:=\left[-1+\epsilon_1,1-\epsilon_1\right]+\left\{2+4k\,|\,|k|\leq 4^Q\right\}.$$
Therefore, there exists a $\delta_1>0$ such that 
\begin{equation}\label{eqsp12}
|\muf(y)|^2\geq\delta_1,\quad(y\in A).
\end{equation}

Take now $x\in[-\frac14,\frac34]$ and let $\omega$ be a $(P,Q)$-good path of the forms mentioned in the hypothesis. If $\omega=\underline0$ then 
$x+\omega=x\in A$ and $|\muf(x+\omega)|^2\geq\delta_1$. In the other case $\omega$ has the form: 
$$\omega=4^{n_1}2+\dots 4^{n_2}2+\dots+4^{n_p}2+4^{n_p+1}k,$$
where $0\leq n_1<\dots<n_p$, $1\leq p\leq P$ and $k$ is an integer with base 4 expansion of length $\leq Q$, so $|k|\leq 4^Q$. Using \eqref{eqsp11} we have, by induction:
$$|\muf(x+\omega)|^2\geq |\muf(x)|^2|\muf(\frac{x}{4^{n_1}}+2+4^{n_2-n_1}2+\dots+4^{n_p-n_1}2+4^{n_p+1-n_1}k)|^2\geq$$
$$|\muf(x)|^2|\muf(\frac{x}{4^{n_1}}+2)|^2|\muf(\frac{x}{4^{n_2}}+\frac{2}{4^{n_2-n_1}}+2+4^{n_3-n_2}2+\dots+4^{n_p-n_2}2+4^{n_p+1-n_2}k)|^2\geq$$
$$|\muf(x)|^2|\muf(\frac{x}{4^{n_1}}+2)|^2|\muf(\frac{x}{4^{n_2}}+\frac{2}{4^{n_2-n_1}}+2)|^2\dots
|\muf(\frac{x}{4^{n_{p-1}}}+\frac{2}{4^{n_{p-1}-n_1}}+\dots+\frac{2}{4^{n_{p-1}-n_{p-2}}}+2)|^2\times$$
$$
|\muf(\frac{x}{4^{n_p}}+\frac{2}{4^{n_p-n_1}}+\dots+\frac{2}{4^{n_p-n_p-1}}+2+4k)|^2.$$
We have, when $n_l\geq1$
$$-1+\epsilon_1<-\frac14\leq\frac{x}{4^{n_l}}+\frac{2}{4^{n_l-n_1}}+\dots+\frac{2}{4^{n_l-n_{l-1}}}\leq \frac{3}{16}+\frac24\frac{1}{1-\frac14}=\frac{41}{48}<1-\epsilon_1.$$
If $n_l=0$ then $l=1$ and $-1+\epsilon_1<\frac{x}{4^0}\leq\frac34<1-\epsilon_1.$ Thus we can use \eqref{eqsp12} on each term in the product above, and we obtain that 
$$|\muf(x+\omega)|^2\geq \delta_1^p\geq\delta_1^P.$$
This proves Lemma \ref{lemsp5}.
\end{proof}

\begin{proof}[Proof of Theorem \ref{thsp2}]

We will show that the conditions of Lemma \ref{lemsp4} are satisfied. 
Take $y\in[-\frac14,\frac54]$ and, take $d_0\dots d_{N-1}$ to be a vertex in the binary tree $\mathcal T(\mathcal L)$.

We distinguish two cases:

Case I: $y\in[-\frac14,\frac34]$. We will construct a path $\lambda$ in the tree starting from the vertex $d_0\dots d_{N-1}$. For this we follow the even-labeled branches until we reach the first $2$ (recall that exactly one of the branches from every vertex is labeled by 0 or 2). If we cannot find a 2, then this means that $\lambda=\underline0$ is a path in the tree from the vertex $d_0\dots d_{N-1}$, and with Lemma \ref{lemsp5}, we obtain $|\muf(y+\lambda)|^2=|\muf(y)|^2\geq\delta$. 

Suppose we can find a $2$ after finitely many steps from $d_0\dots d_{N-1}$. Then from the vertex $d_0\dots d_{N-1}0\dots02$, by hypothesis, we can find a $(P,Q)$-good path $\gamma$ in the tree. Then 
$\lambda:=0\dots02\gamma$ is a $(P+1,Q)$-good path in the tree from the vertex $d_0\dots d_{N-1}$. Then with Lemma \ref{lemsp5}, 
$|\muf(y+\lambda)|^2\geq\delta$.

Case II: $y\in[\frac34,\frac54]$. We will construct a path $\lambda$ from the vertex $d_0\dots d_{N-1}$. For this we follow the odd-labeled branches until we reach the first 1. If we cannot find a 1, then this means that $\lambda=\underline 3$ is a path in the tree from the vertex $d_0\dots d_{N-1}$; so $\lambda=-1$, and $y+\lambda=y-1\in[-\frac14,\frac14]$ so we get $|\muf(y+\lambda)|^2\geq\delta$.

If we can find a 1 after finitely may steps from $d_0\dots d_{N-1}$, then from the vertex $d_0\dots d_{N-1}3\dots 31$ there exists a $(P,Q)$-good path $\gamma$ in the tree. Then take $\lambda:=3\dots31\gamma$, with $p$ $3$s in the beginning. 
Then 
$$y+\lambda=y+3+4\cdot 3+\dots +4^{p-1}3+4^p1+4^{p+1}\gamma=y+4^p-1+4^p+4^{p+1}\gamma=y-1+4^p(2+4\gamma).$$
But then $y-1\in[-\frac14,\frac14]$ and $4^p(2+4\gamma)$ is a $(P+1,Q)$-good path (it is not a path in the tree but that does not matter), that contains at least a $2$ (on position $p$). Therefore, with Lemma \ref{lemsp5}, we get 
$|\muf(y+\lambda)|^2\geq\delta$. 

Thus the hypotheses of Lemma \ref{lemsp4} are satisfied and this implies that $\Lambda(\mathcal L)$ is a spectrum for $\mu_4$.
\end{proof}

As a special consequence of Theorem \ref{thsp2} we obtain the following corollary, which generalizes the results from \cite{JoPe98}, where the labels allowed were only $\{0,1\}$.

\begin{corollary}\label{corsp6}
Suppose $\mathcal L$ is a labeling of the binary tree such that for each vertex $v$ in the tree, the two edges that start from $v$ are labeled by either $\{0,1\}$ or $\{0,3\}$. Then $\Lambda(\mathcal L)$ is a spectrum for $\mu_4$. 
\end{corollary}
\begin{proof}
Clearly this is a spectral labeling because for each vertex the path $\underline 0$ starting at $v$ is in the tree.
This is also a $(0,0)$-good path, so the conditions of Theorem \ref{thsp2} are satisfied.
\end{proof}

\section{Other digits}\label{oth}
In this section, we consider the spectral labeling of the binary tree with other digits, not necessarily $\{0,1,2,3\}$. We show that a spectral labeling is a spectrum if the set of digits is uniformly bounded and the zero label is included at each vertex (partially improving a result in \cite{Str00}). Moreover, we provide the first counterexample for the fractal measure $\mu_4$ of a maximal set of orthogonal exponentials which is {\it not} a spectrum for $\mu_4$.  
\begin{definition}\label{defo1}
Suppose now we want to label the edges in the binary tree with other digits, not necessarily $\{0,1,2,3\}$. At each branching we use different digits, but we obey the rule that at each branching we can use only labels of the type $\{0,a\}$ where $a\in\bz$ is some odd number which varies from one branching to another. Thus, at the root we have a set $A_\ty$ of the form $\{0,a\}$ with $a\in\bz$ odd, and inductively, at each vertex $a_0\dots a_{k-1}$ with $a_0\in A_\ty,\dots,a_{k-1}\in A_{a_0\dots a_{k-2}}$, we have a set $A_{a_0\dots a_{k-1}}$ of the form $\{0,a(a_0,\dots,a_{k-1})\}$ with $a(a_0\dots,a_{k-1})\in\bz$ odd. 
We define the set 
\begin{equation}\label{eqo1}
\Lambda:=\left\{\sum_{k=0}^n 4^ka_k \,|\, a_0\in A_\ty,\dots, a_k\in A_{a_0\dots a_{k-1}}, n\geq0\right\}.
\end{equation}
\end{definition}

\begin{definition}\label{defo2}
Suppose the sets $A_\ty,\dots,A_{a_0\dots a_{k-1}}$ are given as in Definition \ref{defo1}. We say that an integer $\lambda$ has a {\it modified base 4 expansion with digits in $A$} if there exists an infinite sequence $a_0a_1\dots$ with the following properties 
\begin{enumerate}
\item $a_0\in A_\ty$, $a_k\in A_{a_0\dots a_{k-1}}$, for all $k\geq 1$;
\item $\sum_{k=0}^{n-1}a_k 4^k\equiv\lambda\mod 4^{n}$, for all $n\geq 0$.
\end{enumerate}
We call $a_0a_1\dots$ the {\it $A$-base 4 expansion} of $\lambda$. We denote by $\Lambda(A)$ the set of all integers that have a modified base 4 expansion with digits in $A$. 
\end{definition}

\begin{remark}\label{remo1}
The $A$-base 4 expansion is unique. Indeed if $a_0a_1\dots$ and $a_0'a_1'\dots$ are two $A$-base 4 expansions for the same integer $\lambda$, then if they are different, take $n$ to be the first index such that $a_n\neq a_n'$. Then $\sum_{k=0}^na_k4^k\equiv\lambda \equiv\sum_{k=0}^na_k'4^k\mod 4^{n+1}$, but this implies that $a_n\equiv a_n'$, a contradiction, since $a_n,a_n'\in A_{a_0\dots, a_{n-1}}$ and $a_n\neq a_n'$.

Note that if $\mathcal L$ is a spectral labeling and $\lambda\in\mathcal L$, then its base 4 expansion coincides with the $\mathcal L$-base 4 expansion.

\end{remark}

\begin{theorem}\label{tho1}
Consider the sets of digits $A$ as in Definition \ref{defo1}.
\begin{enumerate}
\item
For the set $\Lambda$ in \eqref{eqo1}, the exponentials $\{e_\lambda\,|\,\lambda\in\Lambda\}$ form an orthogonal family. 
There exists a unique spectral labeling $\mathcal L$ such that $\Lambda\subset\Lambda(\mathcal L)$. Moreover $\Lambda(\mathcal L)=\Lambda( A)$.
\item If the sets $A_{a_0\dots a_k}$ are uniformly bounded, then $\Lambda(\mathcal L)$ is a spectrum for $\mu_4$.
\end{enumerate}
\end{theorem}

\begin{proof}
To see that the exponential in $\{e_\lambda\}_{\lambda\in\Lambda}$ are orthogonal, take $\lambda=\sum_{k=0}^\infty {4^k a_k}$ ,$\lambda'=\sum_{k=0}^\infty 4^ka_k'$ in $\Lambda$, $\lambda\neq\lambda'$, $a_k,a_k'=0$ for $k$ large. Let $n$ be the first index such that $a_n\neq a_n'$. Then $\lambda-\lambda'=4^n((a_n-a_n')+4l)$ for some integer $l$. Since $a_n-a_n'$ is odd, we have $\muf(\lambda-\lambda')=0$ (with Lemma \ref{lemma5}). Therefore $e_\lambda\perp e_{\lambda'}$.

Using Zorn's lemma, there is a maximal set $\Lambda'$ of orthogonal exponentials such that $\Lambda\subset\Lambda'$. With Theorem \ref{thma3}, there exists a spectral labeling $\mathcal L$ such that $\Lambda(\mathcal L)=\Lambda'$. 
The key fact here is the uniqueness. We can {\it construct} the spectral labeling $\mathcal L$ as in the proof of Theorem \ref{thma3} and Lemma \ref{lemtwo}. 
We consider base 4 expansions of elements in $\Lambda$. We want to prove that, if we fix $d_0\dots d_{n-1}\in\{0,1,2,3\}$ then the set 
$$D(d_0\dots d_{n-1}):=\{d_n(\lambda)\,|\,\lambda\in\Lambda, d_0(\lambda)=d_0,\dots, d_{n-1}(\lambda)=d_n\}$$
will have 0 or 2 elements, and if it has 2, then they have different parity. 
Since $\Lambda\subset\Lambda'$ it is clear that this set can have at most $2$ elements, and if there are two then they have different parity. So it remains to prove only that it cannot have exactly one. 

Suppose the set contains at least one element. Then there exists $\lambda=\sum_{k=0}^\infty 4^ka_k$, with the digits $a_k$ in the sets $A$, such that the base 4 expansion of $\lambda$ starts with $d_0\dots d_{n-1}$. Take now $\lambda':=\sum_{k=0}^{n-1}4^ka_k+4^na_n$ and $\lambda''=\sum_{k=0}^{n-1}4^ka_k+4^na_n'$ where $a_n'$ is the other digit beside $a_n$ in $A_{a_0\dots a_{n-1}}=\{a_n,a_n'\}$. Since $\lambda-\lambda'$ and $\lambda-\lambda''$ are multiples of $4^n$ the base 4 expansions of $\lambda,\lambda',\lambda''$ will have the same first $n$ digits $d_0\dots d_{n-1}$. The $n+1$-st digits in the base 4 expansion of $\lambda$ and $\lambda'$ will be of different parity because $a_n-a_n'$ is odd. Thus 
$D(d_0\dots d_{n-1})$ has 0 or 2 elements of different parity and these are completely determined from the set $\Lambda$ (not just from the maximal one $\Lambda'$).

Then the construction of the spectral labeling $\mathcal L$ proceeds just as in the proof of Theorem \ref{thma3}.

Next, we prove that an integer $\lambda$ is in $\Lambda(\mathcal L)$ iff it has a modified base 4 expansion with digits in $A$. 
First, we have that an integer $\lambda$ with base 4 expansion $d_0d_1\dots $ is in the tree iff for all $n$, there exists $a_0,\dots, a_N$, $a_0\in A_\ty$, $a_k\in A_{a_0\dots a_{k-1}}$,  such that 
the base 4 expansion of $\sum_{k=0}^Na_k4^k$ begins with $d_0\dots d_{n-1}$. But this implies that 
$\sum_{k=0}^l4^kd_k\equiv \sum_{k=0}^l4^k a_k\mod 4^{l+1}$ for all $l\leq n-1$. In particular the digits $a_0\dots a_{n-1}$ are completely determined by the digits $d_0\dots d_{n-1}$, so they do not change if we increase $n$.

Thus, if $\lambda=d_0d_1\dots$ is in $\Lambda(\mathcal L)$, there exist $a_0,a_1,\dots$ from $A$, such that for all $n\geq0$,
$$\lambda\equiv\sum_{k=0}^n4^kd_k\equiv \sum_{k=0}^n4^ka_k\mod 4^{n+1}.$$
Therefore $\lambda$ is in $\Lambda(A)$. 

Conversely, let $\lambda$ be in $\Lambda(A)$, and let $d_0d_1\dots$ be its base 4 expansion. Then there exist $a_0,a_1,\dots$ from $A$ such that for all $n$.
$$\sum_{k=0}^n4^kd_k\equiv \lambda\equiv\sum_{k=0}^n 4^ka_k\mod 4^{n+1}.$$
This implies that the base 4 expansion of $\sum_{k=0}^n 4^ka_k$ begins with $d_0\dots d_n$ so $d_0\dots d_{n}$ is a label in the tree $\mathcal T(\mathcal L)$, and letting $n\rightarrow\infty$, we get that $\lambda$ is in $\Lambda(\mathcal L)$. This completes the proof of (i).

Next we prove (ii), i.e., if the sets $A_{a_0\dots a_k}$ are uniformly bounded then $\Lambda(\mathcal L)$ is a spectrum. 
We will check the conditions of Theorem \ref{thsp2}. 
Let $Q\geq 0$ such that all the digits $a_k$ used in $\Lambda$ satisfy $|a_k|\leq 4^Q$.

Take a vertex $d_0\dots d_{n-1}$ in the tree $d_i\in\{0,1,2,3\}$. This implies that there exists a $\lambda=\sum_{k=0}^\infty 4^ka_k$ in $\Lambda$, $a_k=0$ for $k$ large, such that the base 4 expansion of $\lambda$ starts with $d_0\dots d_{n-1}$. Take $\lambda':=\sum_{k=0}^{n-1}4^ka_k\in\Lambda$. Since $\lambda-\lambda'=4^nl$ for some integer $l$, the base 4 expansion of $\lambda'$ starts also with $d_0\dots d_{n-1}$. 
But $|\lambda'| \leq \sum_{k=0}^n|a_k|4^k\leq 4^Q\frac{4^n-1}{4-1}\leq 4^{Q+n}$. Therefore the base 4 expansion of $\lambda'$ will have $\underline 0$ or $\underline3$ from position $Q+n$ on. Thus, since $\lambda'\in\Lambda$, there exists a $(0,Q)$-good path in the tree that starts at the vertex $d_0\dots d_{n-1}$. With Theorem \ref{thsp2}, $\Lambda(\mathcal L)$ is a spectrum for $\mu_4$.

\end{proof}

\begin{remark}\label{remstr}
In \cite{Str00}, Strichartz analyzed the spectra of a more general class of measures. When restricted to our example, his results (Theorem 2.7 and 2.8 in \cite{Str00}) cover the case when all vertices at some level $n$ use the same digits $\{0,a_n\}$. In our notation, this means that $A_{a_0,\dots,a_{n-1}}=:A_n$ depends only on the length $n$, and not on the digits $a_0\dots a_{n-1}$. In \cite[Theorem 2.8]{Str00}, an extra condition is needed to guarantee that the set 
$$\Lambda=\left\{\sum_{k=0}^nb_k4^k\,|\, b_k\in\{0,a_k\}\,n\geq0\right\}$$
is a spectrum $\mu_4$. The condition requires the set $\frac{1}{4^{n}}A_0+\frac{1}{4^{n-1}}A_1+\dots+\frac{1}{4}A_{n-1}$ be separated from the zeroes of the function
$$\prod_{k=1}^n\cos^2\left(2\pi\frac{x}{4^k}\right)$$
uniformly in $k$.

Theorem \ref{tho1} improves this result by removing this extra condition. Even when the condition is not satisfied we still get a spectrum for $\mu_4$, namely $\Lambda(A)$, but this might be bigger than $\Lambda$. 
\end{remark}

\begin{example}
Let all the sets $A_{a_0\dots a_{k-1}}$ in Definition \ref{defo1} be equal to $\{0,3\}$. The results in \cite{Str00} do not apply (since $\sum_{k=0}^n\frac{3}{4^k}$ approaches $1$). Then the set 
$$\Lambda=\left\{\sum_{k=0}^n a_k 4^k\,|\, a_k\in \{0,3\}, n\geq0\right\},$$
will give an incomplete set of exponentials. To complete it one has to consider the set $\Lambda(A)$ which in this case 
$$\Lambda(A)=\Lambda\bigcup \left\{\sum_{k=0}^n a_k4^k-4^{n+1}\,|\, a_k\in\{0,3\},n\geq0\right\}.$$
The second part comes from the integers with base 4 expansion ending in $\underline3$. The set $\Lambda$ contains only those integers that have a base 4 expansion ending in $\underline0$. $\Lambda(A)$ is a spectrum, by Theorem \ref{tho1}(ii). The reason for the incompleteness of $\Lambda$ is that the integers are not read correctly (perhaps thoroughly is the better word) from the labels $A$.
\end{example}

\begin{example}
Suppose $A_\ty=\{0,15\}$ and $A_{a_0\dots a_{k-1}}=\{0,9\}$ for all $k\geq 1$. Then the set 
$$\Lambda:=\left\{\sum_{k=0}^na_k 4^k\,|\,a_0\in\{0,15\}, a_k\in\{0,9\}\mbox{ for }k\geq 1, n\geq0\right\},$$
does not give a {\it maximal} set of orthogonal exponentials. $e_3$ is perpendicular to all $e_\lambda$, $\lambda\in\Lambda$. Indeed
$3$ has $A$-base 4 expansion $15\,999\dots$, so $3\in\Lambda(A)$, and $\Lambda(A)$ is a spectrum by Theorem \ref{tho1}.
\end{example}

\begin{example}\label{exsp1}
In this example we construct a set of digits $A$ which will give a spectral labeling, which is not a spectrum. Thus we will have $\Lambda=\Lambda(A)=\Lambda(\mathcal L)$ but $\Lambda$ is not a spectrum. The reason for the incompleteness of $\{e_\lambda\,|\,\lambda\in\Lambda\}$ is thus more subtle, the set {\it is } a maximal set of orthogonal exponentials, but it does not span the entire $L^2(\mu_4)$.

Consider the following set
\begin{equation}\label{eqex1}
\Lambda:=\left\{\sum_{k=0}^N4^k(4^{10^{k+2}-k}+1)\delta_k\,|\, \delta_k\in\{0,1\},N\geq0\right\}.
\end{equation}

We will prove the following 
\begin{proposition}\label{propex2}
There exists a spectral labeling $\mathcal L$ such that $\Lambda(\mathcal L)=\Lambda$, so, by Theorem \ref{thma3} the set $\{e_\lambda\,|\, \lambda\in\Lambda\}$ forms a maximal family of orthogonal exponentials. Nonetheless $\Lambda$ is not a spectrum for $\mu_4$. 
\end{proposition}

\begin{proof}

The elements in $\Lambda$ have the form 
\begin{equation}\label{eqex2}
\lambda=\sum_{k=0}^\infty (4^{10^{k+2}}+4^k)\delta_k,
\end{equation}
where $\delta_k\in\{0,1\}$ and $\delta_k=0$ for $k$ larger than some $N\geq0$.

Let $\lambda=d_0d_1\dots$ be the base $4$ expansion of this element. 
Since $\lambda\geq0$ the expansion ends in $\underline0$. Then, note that 
\begin{enumerate}
\item $d_k=1$ iff one of the following two conditions is satisfied:
\begin{itemize}
\item $k$ is not of the form $10^{n+2}$ and $\delta_k=1$;
\item $k=10^{n+2}$ for some $n\geq0$, and $\delta_n=1$ and $\delta_k=0$.
\end{itemize}
\item $d_k=2$ iff $k=10^{n+2}$ for some $n\geq0$, and $\delta_n=1$ and $\delta_k=1$.
\item $d_k=0$ in all other cases.
\end{enumerate}
We construct the spectral labeling $\mathcal L$ as follows: First, we consider the spectral labeling $\mathcal L_0$ where only the labels $\{0,1\}$ are used at each vertex. We build a new binary tree $\mathcal T(\mathcal L_0,\mathcal L)$ with a different kind of labeling. For the vertices we keep the labels from $\mathcal T(\mathcal L_0)$, but we label the edges differently. We will change the labeling $\{0,1\}$ to $\{1,2\}$ at certain vertices. This will be done in the following way: for all $N\geq0$ and for all vertices $\delta_0\dots \delta_{N}$ with $\delta_{N}=1$, in the subtree with root $\delta_0\dots \delta_N$ we will change the labeling at all vertices at level $10^{N+2}$ from $\{0,1\}$ to $\{1,2\}$. So, at a vertex $\delta_0\dots\delta_N\delta_{N+1}\dots\delta_{10^{N+2}-1}$, the edges are labeled $\{1,2\}$ instead of $\{0,1\}$.

The spectral labeling $\mathcal L$ is obtained by relabeling the vertices consistently with the labels of the edges. 

We have to check that $\Lambda(\mathcal L)=\Lambda$. If $\lambda=d_0d_1\dots\in\Lambda(\mathcal L)$, ending in $\underline0$, then we construct a sequence $\delta_0\delta_1\dots$ by reading the labels of the vertices in $\mathcal T(\mathcal L_0,\mathcal L)$ along $\lambda$. Then by construction 
$$\lambda=\sum_{k=0}^\infty 4^kd_k=\sum_{k=0}^\infty (4^{10^{k+2}}+4^k)\delta_k$$
so $\lambda\in\Lambda$. Conversely, if $\delta_0,\dots,\delta_N$ are in $\{0,1\}$ it is clear that the base 4 expansion of 
$\sum_{k=0}^N(4^{10^{k+2}}+4^k)\delta_k$ is in $\Lambda(\mathcal L)$.

The labeling $\mathcal L$ is a spectral labeling because one can end a path in $\underline0$: just follow the zeros in the labeling of the vertices in $\mathcal T(\mathcal L_0,\mathcal L)$.

Next we prove that $\Lambda$ is not a spectrum for $\mu_4$. We will show that 
\begin{equation}\label{eqex3}
\sum_{\lambda\in\Lambda}|\muf(1+\lambda)|^2<1.
\end{equation}

First, let $\lambda=\lambda(\delta_0\dots \delta_N):=\sum_{k=0}^N(4^{10^{k+2}}+4^k)\delta_k$, with $\delta_N=1$, and let $\lambda=d_0d_1\dots$ be the base 4 expansion. Then 
$d_{10^{N+2}}=1$ and $d_k=0$ for $k>10^{N+2}$. Since for $k<N$, we have $10^{k+2}\leq 10^{N+1}$, and $k<10^{N+1}$, we see that 
$d_k=0$ for $10^{N+1}<k<10^{N+2}$. Thus the base $4$ expansion of $\lambda$ ends with a $1$ on position $10^{N+2}$ and $9\cdot 10^{N+1}$ zeros before that.

We use the following notation: for $e_0e_1\dots e_n$, $.e_0e_1\dots e_n:=\frac{e_0}{4}+\dots+\frac{e_n}{4^n}$. Let $m(x):=\cos^2(2\pi x)$. 
Let the base $4$ expansion of $1+\lambda$ be $b_0b_1\dots$. Then $b_0=d_0+1$ and $b_n=d_n$ for all $n\geq 1$. Then $\frac{1+\lambda}{4}\equiv.b_0\mod\bz$, $\frac{1+\lambda}{4^2}\equiv.b_1b_0\mod\bz$
$\dots \frac{1+\lambda}{4^j}\equiv .b_{j-1}\dots b_0$. Since $m\leq1$ we have 
$$|\muf(1+\lambda)|^2=\prod_{j=1}^\infty m\left(\frac{1+\lambda}{4^j}\right)\leq m\left(\frac{1+\lambda}{4^{10^{N+2}+1}}\right).$$
But $\frac{1+\lambda}{4^{10^{N+2}+1}}\equiv y:=.b_{10^{N+2}}\dots b_{10^{N+1}}\dots b_0\mod\bz$. 
But we saw above that $b_{10^{N+2}}=a_{10^{N+2}}=1$ and $b_n=a_n=0$ for $10^{N+1}<n<10^{N+2}$. So $y-\frac14=y-.1$ has at least $9\cdot 10^{N+1}$ zeros after the decimal point.
Therefore $0\leq y-\frac14=y-.1\leq \frac{1}{4^{9\cdot 10^{N+1}}}$.
Then $$m(y)=\cos^2\left(2\pi\left(\frac14+\left(y-\frac{1}4\right)\right)\right)=\sin^2\left(2\pi\left(y-\frac14\right)\right)\leq 4\pi^2\left(y-\frac14\right)^2\leq \frac{4\pi^2}{4^{18\cdot 10^{N+1}}}.$$
Therefore 
$$|\muf(1+\lambda)|^2\leq m\left(\frac{1+\lambda}{4^{10^{N+1}}}\right)=m(y)\leq \frac{4\pi^2}{4^{18\cdot 10^{N+1}}}.$$
Then 
$$\sum_{\lambda\in\Lambda}|\muf(1+\lambda)|^2=\sum_{N=0}^\infty\sum_{\stackrel{\delta_0,\dots,\delta_{N-1}\in\{0,1\}}{\delta_N=1}}|\muf(1+\lambda(\delta_0\dots\delta_N))|^2\leq\sum_{N=0}^\infty 2^{N}\frac{4\pi^2}{4^{18\cdot 10^{N+1}}}<1.$$
 With Lemma \ref{lemsp3}, this shows that $\Lambda$ is not a spectrum for $\mu_4$.
  
\end{proof}
\end{example}

\section{Further remarks}\label{rem}
In this section we describe some basic properties of spectra for the measure $\mu_4$, and we give an example of a spectral labeling which generates a spectrum but does not satisfy the conditions of Theorem \ref{thsp2}.
\begin{proposition}\label{propr1}
\mbox{}
\begin{enumerate}
\item
If $\Lambda_1,\Lambda_2$ are spectra for $\mu_4$, $\Lambda_1,\Lambda_2\subset\bz$, and $e_1,e_2$ are two integers of different parity, then the set $\Lambda:=(4\Lambda_1+e_1)\cup (4\Lambda_2+e_2)$ is a spectrum for $\mu_4$.
\item
If $\Lambda$ is a spectrum for $\mu_4$, $\Lambda\subset\bz$, then there exist $\Lambda_1, \Lambda_2\subset\bz$ and $e_1,e_2$ integers of different parity such that 
\begin{equation}\label{eqr1}
\Lambda=(4\Lambda_1+e_1)\cup(4\Lambda_2+e_2).
\end{equation}
\end{enumerate}
Moreover, for any decomposition of $\Lambda$ as in \eqref{eqr1}, the sets $\Lambda_1,\Lambda_2$ are spectra for $\mu_4$.
\end{proposition}

\begin{proof} (i)
We use Lemma \ref{lemsp3}. We have for $x\in\br$, using Lemma \ref{lemma4}:
$$\sum_{i=1,2}\sum_{\lambda_i\in\Lambda_i}|\muf(x+4\lambda_i+e_i)|^2=\sum_{i=1,2}\sum_{\lambda_i\in\Lambda_i}\cos^2\left(2\pi\frac{x+e_i}4+\lambda_i\right)\left|\muf\left(\frac{x+e_i}4+\lambda_i\right)\right|^2=$$
$$\sum_{i=1,2}\cos^2\left(2\pi\frac{x+e_i}4\right)\sum_{\lambda_i\in\Lambda_i}\left|\muf\left(\frac{x+e_i}4+\lambda_i\right)\right|^2=\sum_{i=1,2}\cos^2\left(2\pi\frac{x+e_i}4\right)=1.$$
For the next to last equality we used the fact that $\Lambda_i$ are spectra and Lemma \ref{lemsp3}. For the last equality we used the fact that $e_1-e_2$ is odd.

(ii) We can assume that $0\in\Lambda$. Otherwise, we work with $\Lambda-\lambda_0$ for some $\lambda_0\in\Lambda$. Then, since $\Lambda$ is a spectrum, by Theorem \ref{thma3} there is a spectral labeling $\mathcal L$ of the binary tree.
Take $e_1$, $e_2$ to be the labels of the edges that start from the root $\ty$, and take $\Lambda_i$ to be the set of integers that correspond to infinite paths in the subtree with root $e_i$. Then it is clear that \eqref{eqr1} is satisfied.

Assume now that $\Lambda$ is decomposed as in \eqref{eqr1}. We want to prove that $\Lambda_1,\Lambda_2$ are spectra. A simple check, that uses Lemma \ref{lemma5}, shows that $\{e_\lambda\,|\,\lambda\in\Lambda_i\}$ is an orthonormal family, for both $i=1,2$. With Lemma \ref{lemsp3} and the computation above we have for all $x\in\br$,
$$1=\sum_{i=1,2}\cos^2\left(2\pi\frac{x+e_i}4\right)\sum_{\lambda_i\in\Lambda_i}\left|\muf\left(\frac{x+e_i}4+\lambda_i\right)\right|^2=:\sum_{i=1,2}\cos^2\left(2\pi\frac{x+e_i}4\right)h_{\Lambda_i}\left(\frac{x+e_i}4+\lambda_i\right).$$

Take now $x\not\in\bz$. From Lemma \ref{lemsp3}, we have $h_{\Lambda_i}\left(\frac{x+e_i}4+\lambda_i\right)\leq 1$. Also $\cos^2\left(2\pi\frac{x+e_i}4\right)\neq0$ for $i=1,2$. If 
$h_{\Lambda_i}\left(\frac{x+e_i}4+\lambda_i\right)<1$ for one of the $i$'s, then this would contradict the equality above. 
Thus $h_{\Lambda_i}\left(\frac{x+e_i}4+\lambda_i\right)=1$ for all $x\not\in\bz$, $i=1,2$. But as in the proof of Lemma \ref{lemsp3}, this implies that $e_{-x}$ is in the span of $\{e_\lambda\,|\,\lambda\in\Lambda_i\}$ for all  $x\not\in\bz$, and since $e_n$ can be approximated uniformly by $e_x$ with $x\not\in\bz$, it follows that $e_n$ is also spaned by exponentials in $\Lambda_i$. Then as in the proof of Lemma \ref{lemsp3}, it follows that $\Lambda_i$ is a spectrum.
\end{proof}

\begin{remark}\label{remr1}
Suppose $\Lambda_1$ and $\Lambda_2$ are spectra, containing $0$. Let $\mathcal T(\mathcal L_1)$ and $\mathcal T(\mathcal L_2)$ be the spectral labelings of the binary tree that correspond to $\Lambda_1$ and $\Lambda_2$ as in Theorem \ref{thma3}. Let $\{e_1,e_2\}$ be a pair of digits of different parity $e_1,e_2\in\{0,1,2,3\}$. By Proposition \ref{propr1}, $(4\Lambda_1+e_1)\cup(4\Lambda_2+e_2)$ is a new spectrum of $\mu_4$. The corresponding spectral labeling can be obtained by labeling the first two edges, the ones from $\ty$, by $e_1$ and $e_2$, and labeling the edges in the subtree with root $e_1$ using $\mathcal L_1$, and the edges in the subtree with root $e_2$ using $\mathcal L_2$.

Applying Proposition \ref{propr1} several times, we see that the spectral property is a ``tail'' property: it does not depend on the labeling of the first few edges. In other words, if all the subtrees, from some level on, correspond to spectra, then the entire tree will correspond to a spectrum.
\end{remark}

\begin{proposition}\label{propr2}
Let $\mathcal L$ be a spectral labeling. For each vertex $d_0\dots d_{n-1}$, let $\mathcal L_{d_0\dots d_{n-1}}$ be the spectral labeling obtained by reading the labels in the subtree with root $d_0\dots d_{n-1}$. Suppose there exists a finite set $\mathcal S$ of paths in the binary tree $\mathcal T(\mathcal L)$, that start at the root $\ty$, and that satisfy the following conditions:
\begin{enumerate}
\item The paths do not end in $\underline0$ or $\underline3$;
\item For any vertex $d_0\dots d_{n-1}$ that does not lie on any of the paths in $\mathcal S$, the spectral labeling $\mathcal L_{d_0\dots d_{n-1}}$ gives a spectrum, i.e. $\Lambda(\mathcal L_{d_0\dots d_{n-1}})$ is a spectrum.
\end{enumerate}
Then $\Lambda(\mathcal L)$ is a spectrum.
\end{proposition}

\begin{proof}
Let $m(x):=\cos^2(2\pi x)$, $x\in\br$.

Fix $x\in\br$ and let $\omega_0\omega_1\dots$ be a path in $\mathcal T(\mathcal L)$ that does not end in $\underline0$ or $\underline3$. We prove that 
\begin{equation}\label{eqr2}
\lim_{n\rightarrow\infty}\prod_{j=1}^nm\left(\frac{x+\omega_0+\dots+4^{n-1}\omega_{n-1}}{4^j}\right)=0.
\end{equation}
To prove \eqref{eqr2}, we will show first that there exists $\epsilon_0>0$ and a subsequence $\{n_p\}_{p\geq0}$ such that 
\begin{equation}\label{eqr3}
\mbox{dist}\left(\frac{x+\omega_0+\dots+4^{n_p-1}\omega_{n_p-1}}{4^{n_p}},\{0,\frac12,1\}\right)\geq\epsilon_0,\quad(p\geq0).
\end{equation}
If not, then 
$$\mbox{dist}\left(\frac{x+\omega_0+\dots+4^{n-1}\omega_{n-1}}{4^n},\{0,\frac12,1\}\right)\rightarrow0,\mbox{ as }n\rightarrow\infty.$$
Take $\epsilon>0$ small $\epsilon<\frac{1}{4^{10}}$. For $n$ large, $y_n:=\frac{x+\omega_0+\dots+4^{n-1}\omega_{n-1}}{4^n}$ is close to 0, $\frac12$ or $1$.

If $|y_n-0|<\epsilon$ then $y_{n+1}=\frac{y_n+\omega_n}4$ is close to either $0$ when $\omega_n=0$, or $\frac{1}{4},\frac{2}{4},\frac{3}{4}$ when $\omega_n=1,2$ or $3$.

If $|y_n-\frac12|<\epsilon$ then $y_{n+1}$ is close to either $\frac18$, $\frac38$, $\frac58$ or $\frac78$, so it cannot be close to $\{0,\frac12,1\}$. 

If $|y_n-1|<\epsilon$ then $y_{n+1}$ is close to $\{0,\frac12,1\}$ only when $\omega_n=3$.

Thus, the only paths that will make $y_n$ stay close to $\{0,\frac12,1\}$, as $n\rightarrow\infty$, are the ones that end in $\underline0$ or $\underline3$.
This proves \eqref{eqr3}.

If \eqref{eqr3} is satisfied then, since $m(y)=1$ only at $0,\frac12$ and $1$, for $y\in(-1/4,5/4)$, there exists some $\delta>0$, with $\delta<1$, such that for all $p\geq0$,
\begin{equation}\label{eqr4}
m\left(\frac{x+\omega_0+\dots+4^{n_p-1}\omega_{n_p-1}}{4^{n_p}}\right)\leq\delta.
\end{equation}
Then for $n\geq n_p$ we have, since $0\leq m\leq 1$ and $m$ is $\bz$-periodic,
$$\prod_{k=1}^nm\left(\frac{x+\omega_0+\dots+4^{n-1}\omega_{n-1}}{4^j}\right)\leq\prod_{l=1}^pm\left(\frac{x+\omega_0+\dots+4^{n_l-1}\omega_{n_l-1}}{4^{n_l}}\right)\leq\delta^p.$$

This implies \eqref{eqr2}.

Let $\mathcal V(\mathcal S)$ be the set of labels of vertices on the paths in $\mathcal S$.

To prove Proposition \ref{propr2}, we use Lemma \ref{lemsp3}. Using the computation in the proof of Proposition \ref{propr1} we have for all $n\geq0$:
$$\sum_{\lambda\in\Lambda(\mathcal L)}|\muf(x+\lambda)|^2=\sum_{d_0\dots d_{n-1}\in\Lambda(\mathcal L)}\prod_{j=1}^nm\left(\frac{x+d_0+\dots+4^{n-1}d_{n-1}}{4^j}\right)\times$$$$\sum_{\lambda\in\Lambda(\mathcal L_{d_0\dots d_{n-1}})}\left|\muf\left(\frac{x+d_0+\dots+4^{n-1}d_{n-1}}{4^n}+\lambda\right)\right|^2\geq $$
$$\sum_{d_0\dots d_{n-1}\in\Lambda(\mathcal L)\setminus\mathcal V(\mathcal S)}\prod_{j=1}^nm\left(\frac{x+d_0+\dots+4^{n-1}d_{n-1}}{4^j}\right)\sum_{\lambda\in\Lambda(\mathcal L_{d_0\dots d_{n-1}})}\left|\muf\left(\frac{x+d_0+\dots+4^{n-1}d_{n-1}}{4^n}+\lambda\right)\right|^2=(\ast).$$
Since $\Lambda(\mathcal L_{d_0\dots d_{n-1}})$ is a spectrum for all $d_0\dots d_{n-1}$ not in $\mathcal V(\mathcal S)$, with Lemma \ref{lemsp3} we obtain
$$(\ast)=\sum_{\stackrel{d_0\dots d_{n-1}}{\in\Lambda(\mathcal L)\setminus\mathcal V(\mathcal S)}}\prod_{j=1}^nm\left(\frac{x+d_0+\dots+4^{n-1}d_{n-1}}{4^j}\right)=1-\sum_{\stackrel{d_0\dots d_{n-1}}{\in\Lambda(\mathcal L)\cap\mathcal V(\mathcal S)}}\prod_{j=1}^nm\left(\frac{x+d_0+\dots+4^{n-1}d_{n-1}}{4^j}\right)=(\ast\ast).$$
We used \eqref{eqsp4} for the previous equality.

We use the notation $\omega=d_0(\omega)d_1(\omega)\dots$. We have then with \eqref{eqr2},

$$(\ast\ast)=1-\sum_{\omega\in\mathcal S}\prod_{j=1}^nm\left(\frac{x+d_0(\omega)+\dots+4^{n-1}d_{n-1}(\omega)}{4^j}\right)\rightarrow1.$$
\end{proof}

\begin{example}\label{exr4}
We construct an example of a spectral labeling $\mathcal L$ such that $\Lambda(\mathcal L)$ is a spectrum for $\mu_4$ but $\mathcal L$ does not satisfy the conditions of Theorem \ref{thsp2}. 
\begin{figure}[ht]\label{fig2}
\centerline{ \vbox{\hbox{\epsfxsize 5cm\epsfbox{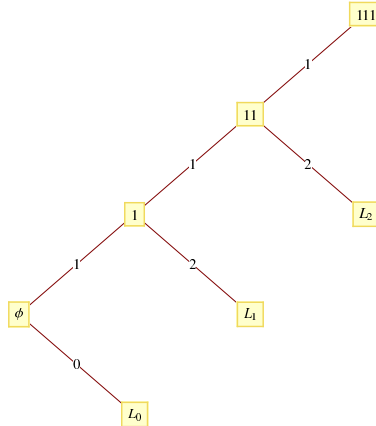}}}}
\caption{A spectral labeling which gives a spectrum but does not satisfy the conditions of Theorem \ref{thsp2}.}
\end{figure}

For this pick an infinite path in the binary tree and label it with $111\dots$.

Let $\mathcal L_0$ be the spectral labeling which uses $\{0,1\}$ at each branch. We know $\Lambda(\mathcal L_0)$ is a spectrum. 
Let $\mathcal L_n$ be the spectral labeling which uses $\{1,2\}$ for first $n$ levels in the tree and $\{0,1\}$ for the rest. Using Proposition \ref{propr1}, we have that $\Lambda(\mathcal L_n)$ is a spectrum. 

We label the edges in the binary tree as follows. At the root, we already have one label $1$. We use $0$ for the other edge, and we label the subtree with root $0$ using $\mathcal L_0$. At the vertex $\underbrace{1\dots1}_{n\mbox{ times}}$, we already have one label $1$. We use $2$ for the other edge, and we label the subtree with root $\underbrace{1\dots1}_{n\mbox{ times}}$ using $\mathcal L_n$. 

Doing this for all $n$, we get a spectral labeling $\mathcal L$. Proposition \ref{propr2} shows that $\Lambda(\mathcal L)$ is a spectrum for $\mu_4$. 

Clearly $\mathcal L$ does not satisfy the conditions of Theorem \ref{thsp2}, because for any $P\geq0$, if we take the vertex $\underbrace{1\dots1}_{P+1\mbox{ times}}$, any path from this vertex has to go through a barrage of at least $P+1$ twos, before it can end in $\underline 0$.

\end{example}

\begin{acknowledgements}
We would like to thank professors Palle Jorgensen, Keri Kornelson, Judith Packer, Gabriel Picioroaga, Yang Wang and Eric Weber for helpful discussions and suggestions.
\end{acknowledgements}
\bibliographystyle{alpha}
\bibliography{stspec}

\end{document}